\documentclass[11pt]{amsart}
\usepackage{amsmath}
\usepackage{amsfonts}
\usepackage[active]{srcltx}

\newtheorem{theorem}{Theorem}

\newtheorem{lemma}{Lemma}

\newtheorem{corollary}{Corollary}

\newtheorem*{Go}{Theorem G}
\newtheorem*{Go2}{Theorem G2}
\newtheorem*{tep}{Theorem T}

\begin{document}
\author{Tsitsino tepnadze}
\title[Approximation Properties of Ces\`{a}ro Means]{On the Approximation
Properties of Ces\`{a}ro Means of Negative Order of double Vilenkin-Fourier
Series}
\address{T.Tepnadze, Department of Mathematics, Faculty of Exact and Natural
Sciences, Ivane Javakhishvili Tbilisi State University, Chavcha\-vadze str.
1, Tbilisi 0128, Georgia}
\email{tsitsinotefnadze@gmail.com}
\maketitle

\begin{abstract}
In this paper we establish approximation properties of Ces\`{a}ro $%
(C,-\alpha ,-\beta )$ means with $\alpha ,\beta $ $\epsilon $ $(0,1)$ of
Vilenkin-Fourier series.This result allows one to obtain a condition which
is sufficient for the convergence of the means $\sigma _{n,m}^{-\alpha
,-\beta }(x,y,f)$ to $f(x,y)$ in the $L^{p}-$metric.
\end{abstract}

\footnotetext{%
2010 Mathematics Subject Classification 42C10 .
}

Let $N_{+}$ denote the set of positive integers, $N:=N_{+}\cup \{0\}.$ Let $%
m:=\left( m_{0},m_{1},...\right) $ denote a sequence of positive integers
not lass then 2. Denote by $Z_{m_{k}}:=\{0,1,...,m_{k}-1\}$ the additive
group of integers modulo $m_{k}$. Define the group $G_{m}$ as the complete
direct product of the groups $Z_{m_{j}}$, with the product of the discrete
topologies of $Z_{mj}$'s.

The direct product of the measures%
\begin{equation*}
\mu _{k}\left( \{j\}\right) :=\frac{1}{m_{k}}\ \ \ \ \ \ \ \ \left( j\text{ }%
\in Z_{m_{k}}\right)
\end{equation*}%
is the Haar measure on $G_{m}$ with $\mu \left( G_{m}\right) =1.$ If the
sequence $m$ is bounded, then $G_{m}$\ is called a bounded Vilenkin group.
In this paper we will consider only bounded Vilenkin group. The elements of $%
G_{m}$ can be represented by sequences $x:=\left(
x_{0},x_{1},...,x_{j},...\right) ,$ $\left( x_{j}\in Z_{m_{j}}\right) .$ The
group operation $+$ in $G_{m}$ is given by%
\begin{equation*}
x+y=\left( \left( x_{0}+y_{0\text{ }}\right) mod\text{ }m_{0},...,\left(
x_{k}+y_{k\text{ }}\right) mod\text{ }m_{k},...\right) ,
\end{equation*}%
where $x:=\left( x_{0},...,x_{k},...\right) $ and $y:=\left(
y_{0},...,y_{k},...\right) \in G_{m}.$ The inverse of $+$ will be denoted by 
$-.$

It is easy to give a base for the neighborhoods of $G_{m}:$

\begin{equation*}
I_{0}\left( x\right) :=G_{m},
\end{equation*}

\begin{equation*}
I_{n}\left( x\right) :=\{y\in G_{m}|y_{0\text{ }}=x_{0\text{ }},...,y_{n-1%
\text{ }}=x_{n-1\text{ }}\}
\end{equation*}%
for $x$ $\in $ $G_{m},$ $n$ $\in $ $N.$ Define $I_{n}:=I_{n}\left( 0\right) $
for $n\in N_{+}$. Set $e_{n}:=\left( 0,...,0,1,0,...\right) \in $ $G_{m}$
the $n+1$ th coordinate of which is $1$ and the rest are zeros $\left( n\in
N\right) .$

If we define the so-called generalized number system based on $m$ in the
following way: $M_{0}:=1,$ $M_{k+1}:=m_{k}M_{k}$ $\left( k\in N\right) ,$
then every $n$ $\in $ $N$ can be uniquely expressed as $n=\sum%
\limits_{j=0}^{\infty }n_{j}M_{j},$ where $n_{j}$ $\in \ Z_{m_{j}}$ $\left(
j\in N_{+}\right) $ and only a finite number of $n_{j}$'s differ from zero.
We use the following notation. Let $\left\vert n\right\vert :=$max$\{k\in
N:n_{k}\neq 0\}$ (that is , $M_{|n|}\leq n<M_{|n|+1}$).

Next, we introduce of $G_{m}$ an orthonormal system which is called Vilenkin
system. At first define the complex valued functions $r_{k}\left( x\right)
:G_{m}\rightarrow C.$ the generalized Rademacher functions in this way 
\begin{equation*}
r_{k}(x):=\exp \frac{2\pi ix_{k}}{m_{k}}\text{\ \ \ }\left( i^{2}=-1,\text{ }%
x\in G_{m},\text{ }k\text{ }\in \text{ }N\right) .
\end{equation*}

Now define the Vilenkin system $\psi :=\left( \psi _{n}:n\in N\right) $ on $%
G_{m}$ as follows.

\begin{equation*}
\psi _{n}\left( x\right) :=\prod\limits_{k=0}^{\infty }r_{k}^{n_{k}}\left(
x\right) ,\text{\ \ \ \ \ \ \ }\left( n\text{ }\epsilon \text{ }N\right) .
\end{equation*}

In particular, we call the system the Walsh-Paley if $m=2.$

The Dirichlet kernels is defined by

\begin{equation*}
D_{n}:=\sum\limits_{k=0}^{n-1}\psi _{k},\ \ \ \ \ \left( n\text{ }\in
N_{+}\right) .
\end{equation*}

Recall that (see \cite{Gol} or \cite{Sw})

\begin{equation}
\quad \hspace*{0in}D_{M_{n}}\left( x\right) =\left\{ 
\begin{array}{l}
\text{ }M_{n},\text{\thinspace \thinspace \thinspace \thinspace if\thinspace
\thinspace }x\in I_{n}, \\ 
\text{ }0,\text{\thinspace \thinspace \thinspace \thinspace \thinspace if
\thinspace \thinspace }x\notin I_{n}.%
\end{array}%
\right.  \label{for1}
\end{equation}

The Vilenkin system is orthonormal and complete in \ $L^{1}\left(
G_{m}\right) $\cite{AVDR}.

Next, we introduce some notation with respect to the theory of
two-demonsional Vilenkin system. Let $\tilde{m}$ be a sequence like $m$. The
relation between the sequences $\left( \tilde{m}_{n}\right) $ and \ $\left( 
\tilde{M}_{n}\right) $ is the same as between sequences $\left( m_{n}\right) 
$ and $\left( M_{n}\right) .$ The group $G_{m}\times G_{\tilde{m}}$ is
called a two-dimensional Vilenkin group. The normalized Haar measure is
denoted by $\mu $ as in the one-dimensional case. We also suppose that $m=%
\tilde{m}$ and $G_{m}\times G_{\tilde{m}}=G_{m}^{2}.$

The norm of the space $L^{p}\left( G_{m}^{2}\right) $ is defined by

\begin{equation*}
\left\Vert f\right\Vert _{p}:=\left( \int\limits_{ G_{m}^{2}}\left\vert f\left( x,y\right) \right\vert ^{p}d\mu \left( x,y\right)
\right) ^{1/p},\ \ \ \left( 1\leq p<\infty \right) .
\end{equation*}

Denote by $C\left( G_{m}^{2}\right) $ the class of continuous functions on
the group $G_{m}^{2}$, endoved with the supremum norm.

For the sake of brevity in notation, we agree to write $L^{\infty }\left(
G_{m}^{2}\right) $ instead of $C\left( G_{m}^{2}\right) .$

The two-dimensional Fourier coefficients,the rectangular partial sums of the
Fourier series,the Dirichlet kernels with respect to the two-dimensional
Vilenkin system are defined as follow:

\begin{equation*}
\widehat{f}\left( n_{1},n_{2}\right) :=\int\limits_{G_{m}^{2}}f\left(
x,y\right) \bar{\psi}_{n_{1}}\left( x\right) \bar{\psi}_{n_{2}}\left(
y\right) d\mu \left( x,y\right) ,
\end{equation*}

\begin{equation*}
S_{n_{1},n_{2}}\left( x,y,f\right)
:=\sum\limits_{k_{1}=0}^{n_{1}-1}\sum\limits_{k_{2}=0}^{n_{2}-1}\widehat{f}%
\left( k_{1},k_{2}\right) \psi _{k_{1}}\left( x\right) \psi _{k_{2}}\left(
y\right) ,\ 
\end{equation*}

\begin{equation*}
D_{n_{1},n_{2}}\left( x,y\right) :=D_{n_{1}}\left( x\right) D_{n_{2}}\left(
y\right) ,
\end{equation*}%
\ \ 

Denote

\begin{equation*}
S_{n}^{\left( 1\right) }\left( x,y,f\right) :=\sum\limits_{l=0}^{n-1}%
\widehat{f}\left( l,y\right) \psi _{l}\left( x\right) ,\ 
\end{equation*}

\begin{equation*}
S_{m}^{\left( 2\right) }\left( x,y,f\right) :=\sum\limits_{r=0}^{m-1}%
\widehat{f}\left( x,r\right) \psi _{r}\left( y\right) ,\ 
\end{equation*}

where

\begin{equation*}
\widehat{f}\left( l,y\right) =\int\limits_{G_{m}}f\left( x,y\right) \psi
_{l}\left( x\right) d\mu \left( x\right)
\end{equation*}

and

\begin{equation*}
\widehat{f}\left( x,r\right) =\int\limits_{G_{m}}f\left( x,y\right) \psi
_{r}\left( y\right) d\mu \left( y\right) .
\end{equation*}

The $(c,-\alpha ,-\beta )$ means of the two-dimensional Vilenkin-Fourier
series are defined as

\begin{equation*}
\sigma _{n,m}^{-\alpha ,-\beta }\left( x,y,f\right) =\frac{1}{A_{n}^{-\alpha
}A_{m}^{-\beta }}\sum\limits_{i=0}^{n}\sum\limits_{j=0}^{m}A_{n-i}^{-\alpha
}A_{m-j}^{-\beta }\hat{f}\left( i,j\right) \text{ }\psi _{i}\left( u\right)
\psi _{j}\left( v\right) ,
\end{equation*}

where

\begin{equation*}
A_{0}^{\alpha }=1,\ \ \ \ \ \ \ A_{n}^{\alpha }=\frac{\left( \alpha
+1\right) ...\left( \alpha +n\right) }{n!}.
\end{equation*}

It is well Known that \cite{Zy}

\begin{equation}
A_{n}^{\alpha }=\sum\limits_{k=0}^{n}A_{k}^{\alpha -1}.\text{ \ \ \ \ \ \ }
\label{for2}
\end{equation}

\begin{equation}
A_{n}^{\alpha }-A_{n-1}^{\alpha }=A_{n}^{\alpha -1}.\ \ \ \ \ \ \ 
\label{for3}
\end{equation}

\begin{equation}
A_{n}^{\alpha }\sim n^{\alpha }.\ \ \ \ \ \ \   \label{for4}
\end{equation}

The dyadic partial moduli of continuity of a function $f\in L^{p}\left(
G_{m}^{2}\right) \ $\ in the $L^{p}$-norm are defined by

\begin{equation*}
\omega _{1}\left( f,\frac{1}{M_{n}}\right) _{p}=\sup_{u\in I_{n}}\left\Vert
f\left( \cdot -u,\cdot \right) -f\left( \cdot ,\cdot \right) \right\Vert
_{p},
\end{equation*}

\begin{equation*}
\omega _{2}\left( f,\frac{1}{M_{n}}\right) _{p}=\sup_{v\in I_{n}}\left\Vert
f\left( \cdot ,\cdot -v\right) -f\left( \cdot ,\cdot \right) \right\Vert
_{p},
\end{equation*}

while the dyadic mixed modulus of continuity is defined as follows:

\begin{equation*}
\omega _{1,2}\left( f,\frac{1}{M_{n}},\frac{1}{M_{m}}\right) _{p}
\end{equation*}

\begin{equation*}
=\sup_{\left( u,v\right) \in I_{n}\times I_{m}}\left\Vert f\left( \cdot
-u,\cdot -v\right) -f\left( \cdot -u,\cdot \right) -f\left( \cdot ,\cdot
-v\right) +f\left( \cdot ,\cdot \right) \right\Vert _{p},
\end{equation*}

it is clear that 
\begin{equation*}
\omega _{1,2}\left( f,\frac{1}{M_{n}},\frac{1}{M_{m}}\right) _{p}\leq \omega
_{1}\left( f,\frac{1}{M_{n}}\right) _{p}+\omega _{2}\left( f,\frac{1}{M_{m}}%
\right) _{p}.
\end{equation*}

The dyadic total modulus of continuity is defined by

\begin{equation*}
\omega \left( f,\frac{1}{M_{n}}\right) _{p}=\sup_{\left( u,v\right) \in
I_{n}\times I_{n}}\left\Vert f\left( \cdot -u,\cdot -v\right) -f\left( \cdot
,\cdot \right) \right\Vert _{p}
\end{equation*}

The problems of summability of partial sums and Ces\`{a}ro \ means for
Walsh-Fourier series were studied in \cite{Fi}, \cite{GoAMH}-\cite{Su},\cite%
{Tev}. In his monography \cite{Zh} Zhizhinashvili investigated the behavior
of Ces\`{a}ro method of negative order for trigonometric Fourier series in
detail. Goginava \cite{GoJAT} studied analogical question in case of the
Walsh system. In particular, the following theorem is proved.

\begin{Go}
\cite{GoJAT}Let $f$ belong to $L^{p}\left( G_{2}\right) $ for some $p$ $\in $
$\left[ 1,\infty \right] $ and \ $\alpha $ $\in $ $\left( 0,1\right) $. Then
for any \ $2^{k}\leq n<2^{k+1}$ $(k,n\in N)$\ the inequality
\end{Go}

\begin{equation*}
\left\Vert \sigma _{n}^{-\alpha }\left( f\right) -f\right\Vert _{p}\leq
c\left( p,\alpha \right) \left\{ 2^{k\alpha }\omega \left(
1/2^{k-1},f\right) _{p}+\sum\limits_{r=0}^{k-2}2^{r-k}\omega \left(
1/2^{r},f\right) _{p}\right\}
\end{equation*}

holds true.

The present author in \cite{Te} investigated analogous question in the case
of Vilenkin system.

\begin{tep}
\label{T2}Let $f$ belong to $L^{p}\left( G_{m}\right) $ for some $p$ $\in $ $%
\left[ 1,\infty \right] $ and \ $\alpha $ $\in $ $\left( 0,1\right) $. Then
for any \ $M_{k}\leq n<M_{k+1}$ $(k,n\in N)$\ the inequality

\begin{equation*}
\left\Vert \sigma _{n}^{-\alpha }\left( f\right) -f\right\Vert _{p}\leq
c\left( p,\alpha \right) \left\{ M_{k}^{\alpha }\omega \left(
1/M_{k-1},f\right) _{p}+\sum\limits_{r=0}^{k-2}\frac{M_{r}}{M_{k}}\omega
\left( 1/M_{r},f\right) _{p}\right\}
\end{equation*}

holds true.
\end{tep}

Gognava in \cite{GogAn} studied approximation properties of Ces\`{a}ro $%
(c,-\alpha ,-\beta )$ means with $\alpha ,\beta $ $\in $ $\left( 0,1\right) $
question in the case of double Walsh-Furier series.The following theorem was
proved.

\begin{Go2}
Let $f$ belong to $L^{p}\left(  G_{2}^{2}\right) $ for some $p$ $\in $
$\left[ 1,\infty \right] $ and \ $\alpha ,\beta $ $\in $ $\left( 0,1\right) $%
. Then for any \ $2^{k}\leq n<2^{k+1},2^{l}\leq m<2^{l+1}$ $(k,n\in N)$\ the
inequality

\begin{equation*}
\left\Vert \sigma _{n,m}^{-\alpha ,-\beta }\left( f\right) -f\right\Vert
_{p}\leq c\left( \alpha ,\beta \right) \left( 2^{k\alpha }\omega _{1}\left(
f,1/2^{k-1}\right) _{p}+2^{l\beta }\omega _{2}\left( f,1/2^{l-1}\right)
_{p}+\right.
\end{equation*}

\begin{equation*}
+2^{k\alpha }2^{l\beta }\omega _{1,2}\left( f,1/2^{k-1},1/2^{l-1}\right)
_{p}+
\end{equation*}

\begin{equation*}
\left. +\sum\limits_{r=0}^{k-2}2^{r-k}\omega _{1}\left( f,1/2^{r}\right)
_{p}+\sum\limits_{s=0}^{l-2}2^{s-l}\omega _{2}\left( f,1/2^{s}\right)
_{p}\right)
\end{equation*}

holds true.
\end{Go2}

In this paper, we estabilish analogous question in the case of double
Vilenkin-Fouries series.

\begin{theorem}
\label{T3}Let $f$ belong to $L^{p}\left( G_{m}^{2}\right) $ for some $p$ $%
\in $ $\left[ 1,\infty \right] $ and \ $\alpha $ $\in $ $\left( 0,1\right) $%
. Then for any \ $M_{k}\leq n<M_{k+1}$ $M_{l}\leq m<M_{l+1}(k,n,m,l\in N)$\
the inequality
\end{theorem}

\begin{equation*}
\left\Vert \sigma _{n,m}^{-\alpha ,-\beta }\left( f\right) -f\right\Vert
_{p}\leq c\left( \alpha ,\beta \right) \left( \omega _{1}\left(
f,1/M_{k-1}\right) _{p}M_{k}^{\alpha }+\omega _{2}\left( f,1/M_{l-1}\right)
_{p}M_{l}^{\beta }+\right.
\end{equation*}

\begin{equation*}
+\omega _{1,2}\left( f,1/M_{k-1},1/M_{l-1}\right) _{p}M_{k}^{\alpha
}M_{l}^{\beta }+
\end{equation*}

\begin{equation*}
\left. +\sum\limits_{r=0}^{k-2}\frac{M_{r}}{M_{k}}\omega _{1}\left(
f,1/M_{r}\right) _{p}+\sum\limits_{s=0}^{l-2}\frac{M_{s}}{M_{l}}\omega
_{2}\left( f,1/M_{s}\right) _{p}\right)
\end{equation*}

holds true.

\begin{corollary}
\label{C1} Let $f\ $belong to $L^{p}$ for some $p\in \left[ 1,\infty \right]
.$ If

\begin{equation*}
M_{k}^{\alpha }\omega _{1}\left( f,\frac{1}{M_{k}}\right) _{p}\rightarrow 0\
\ as\ k\rightarrow \infty \left( 0<\alpha <1\right) ,
\end{equation*}

\begin{equation*}
M_{l}^{\beta }\omega _{1}\left( f,\frac{1}{M_{l}}\right) _{p}\rightarrow 0%
\text{\ \ as \ }l\rightarrow \infty \ \left( 0<\beta <1\right) ,
\end{equation*}
\end{corollary}

\begin{equation*}
M_{k}^{\alpha }M_{l}^{\beta }\omega _{1},_{2}\left( f,\frac{1}{M_{k}},\frac{1%
}{M_{l}}\right) _{p}\rightarrow 0\text{\ \ \ \ as \ }\ k,l\rightarrow \infty
,
\end{equation*}

then%
\begin{equation*}
\left\Vert \sigma _{n,m}^{-\alpha ,-\beta }\left( f\right) -f\right\Vert
_{p}\rightarrow 0\ \,\ \ \text{as \ }n,m\rightarrow \infty .
\end{equation*}

\begin{corollary}
\label{C2} Let $f\ $belong to $L^{p}$ for some $p\in \left[ 1,\infty \right] 
$ and let $\alpha ,\beta \in \left( 0,1\right) ,$ $\alpha +\beta <1.$ If

\begin{equation*}
\omega \left( f,\frac{1}{M_{n}}\right) _{p}=o\left( \left( \frac{1}{M_{n}}%
\right) ^{\alpha +\beta }\right) ,
\end{equation*}

then%
\begin{equation*}
\left\Vert \sigma _{n,m}^{-\alpha ,-\beta }\left( f\right) -f\right\Vert
_{p}\rightarrow 0\ \,\ \ \text{as \ }n,m\rightarrow \infty .
\end{equation*}
\end{corollary}

The following theorem shows that Corollary 2 cannot be improved.

\begin{theorem}
\label{T4} For every $\alpha ,\beta \in \left( 0,1\right) ,$ $\alpha +\beta
<1,$ there exists a function $f_{0}\in C\left(  G_{m}^{2}\right) $
for which

\begin{equation*}
\omega \left( f,\frac{1}{M_{n}}\right) _{C}=O\left( \left( \frac{1}{M_{n}}%
\right) ^{\alpha +\beta }\right) ,
\end{equation*}

and%
\begin{equation*}
\lim \sup_{n\rightarrow \infty }\left\Vert \sigma _{M_{n},M_{n}}^{-\alpha
,-\beta }\left( f\right) -f\right\Vert _{1}>0.
\end{equation*}
\end{theorem}

In order to prove Theorem \ref{T3} we need the following lemmas

\begin{lemma}
\label{L1}\cite{AVDR}Let $\alpha _{1},...,\alpha _{n}$ be real numbers.Then
\end{lemma}

\begin{equation*}
\frac{1}{n}\int\limits_{G_{m}}\left\vert \sum\limits_{k=1}^{n}\alpha
_{k}D_{k}(x)\right\vert d\mu (x)\leq \frac{c}{\sqrt{n}}\left(
\sum\limits_{k=1}^{n}\alpha _{k}^{2}\right) ^{1/2}.
\end{equation*}

where $c$ is an absolute constant.

\begin{lemma}
\label{L2}Let $f$ $\in L^{p}(G_{m}^{2})$ for some $p$ $\in $ $\left[
1,\infty \right] .$ Then for every $\alpha ,\beta $ $\in $ $\left(
0,1\right) $ the following estimations holds%
\begin{equation*}
I:=\frac{1}{A_{n}^{-\alpha }A_{m}^{-\beta }}\left\Vert \text{ }%
\int\limits_{G_{m}^{2}}\sum\limits_{i=0}^{M_{k-1}-1}\sum%
\limits_{j=0}^{M_{l-1}-1}A_{n-i}^{-\alpha }A_{m-j}^{-\beta }\psi _{i}\left(
u\right) \psi _{j}\left( v\right) \times \right.
\end{equation*}

\begin{equation*}
\left. \times \left[ f\left( \cdot -u,\cdot -v\right) -f\left( \cdot ,\cdot
\right) \right] d\mu \left( u,v\right) \right\Vert _{p}\leq
\end{equation*}

\begin{equation*}
\leq c\left( \alpha ,\beta \right) \left( \sum\limits_{r=0}^{k-1}\frac{M_{r}%
}{M_{k}}\omega _{1}\left( f,1/M_{r}\right) _{p}+\sum\limits_{s=0}^{l-1}\frac{%
M_{s}}{M_{l}}\omega _{2}\left( f,1/M_{s}\right) _{p}\right) ,
\end{equation*}
\end{lemma}

where $M_{k}\leq n<M_{k+1},M_{l}\leq m<M_{l+1}.$\newline

\begin{proof}[Proof of Lemma \protect\ref{L2}]
Applying Abel's transformation, from (\ref{for2}) we get

\begin{equation}
I\leq \frac{1}{A_{n}^{-\alpha }A_{m}^{-\beta }}\left\Vert \text{ }%
\int\limits_{G_{m}^{2}}\sum\limits_{i=1}^{M_{k-1}-1}\sum%
\limits_{j=1}^{M_{l-1}-1}A_{n-i+1}^{-\alpha -1}A_{m-j+1}^{-\beta
-1}D_{i}\left( u\right) D_{j}\left( v\right) \times \right.  \label{for5}
\end{equation}

\begin{equation*}
\left. \times \left[ f\left( \cdot -u,\cdot -v\right) -f\left( \cdot ,\cdot
\right) \right] d\mu \left( u,v\right) \right\Vert _{p}
\end{equation*}

\begin{equation*}
+\frac{1}{A_{n}^{-\alpha }A_{m}^{-\beta }}\left\Vert \text{ }%
\int\limits_{G_{m}^{2}}A_{m-M_{l-1}+1}^{-\beta }D_{M_{l-1}}\left( v\right)
\sum\limits_{i=1}^{M_{k-1}-1}A_{n-i+1}^{-\alpha -1}D_{i}\left( u\right)
\times \right.
\end{equation*}

\begin{equation*}
\left. \times \left[ f\left( \cdot -u,\cdot -v\right) -f\left( \cdot ,\cdot
\right) \right] d\mu \left( u,v\right) \right\Vert _{p}
\end{equation*}

\begin{equation*}
+\frac{1}{A_{n}^{-\alpha }A_{m}^{-\beta }}\left\Vert \text{ }%
\int\limits_{G_{m}^{2}}A_{n-M_{k-1}+1}^{-\alpha }D_{M_{k-1}}\left( u\right)
\sum\limits_{j=1}^{M_{l-1}-1}A_{m-j+1}^{-\beta -1}D_{j}\left( v\right)
\times \right.
\end{equation*}

\begin{equation*}
\left. \times \left[ f\left( \cdot -u,\cdot -v\right) -f\left( \cdot ,\cdot
\right) \right] d\mu \left( u,v\right) \right\Vert _{p}
\end{equation*}

\begin{equation*}
+\frac{1}{A_{n}^{-\alpha }A_{m}^{-\beta }}\left\Vert \text{ }%
\int\limits_{G_{m}^{2}}A_{n-M_{k-1}+1}^{-\alpha }A_{m-M_{l-1}+1}^{-\beta
}D_{M_{k-1}}\left( u\right) D_{M_{l-1}}\left( v\right) \times \right.
\end{equation*}

\begin{equation*}
\left. \times \left[ f\left( \cdot -u,\cdot -v\right) -f\left( \cdot ,\cdot
\right) \right] d\mu \left( u,v\right) \right\Vert
_{p}=I_{1}+I_{2}+I_{3}+I_{4}.
\end{equation*}

From the generalized Minkowski inequality, and by (\ref{for1}) and (\ref%
{for4}) we obtain

\begin{equation}
I_{4}\leq \frac{1}{A_{n}^{-\alpha }A_{m}^{-\beta }}\int\limits_{G_{m}^{2}}%
\left\vert A_{n-M_{k-1}+1}^{-\alpha }A_{m-M_{l-1}+1}^{-\beta
}D_{M_{k-1}}\left( u\right) D_{M_{l-1}}\left( v\right) \right\vert \times
\label{for6}
\end{equation}

\begin{equation*}
\times \left\Vert f\left( \cdot -u,\cdot -v\right) -f\left( x,y\right)
\right\Vert _{p}d\mu \left( u,v\right)
\end{equation*}

\begin{equation*}
\leq c\left( \alpha ,\beta \right) M_{k-1}M_{l-1}\int\limits_{I_{k-1}\times
I_{l-1}}\left\Vert f\left( \cdot -u,\cdot -v\right) -f\left( \cdot ,\cdot
\right) \right\Vert _{p}d\mu \left( u,v\right)
\end{equation*}

\begin{equation*}
=O\left( \omega _{1}(f,1/M_{k-1})_{p}+\omega _{2}(f,1/M_{l-1})_{p}\right) .
\end{equation*}%
It is evident that

\begin{equation}
I_{1}\leq \frac{1}{A_{n}^{-\alpha }A_{m}^{-\beta }}\text{ }%
\sum\limits_{r=0}^{k-2}\sum\limits_{s=0}^{l-2}\left\Vert \text{ }%
\int\limits_{G_{m}^{2}}\sum\limits_{i=M_{r}}^{M_{r+1}-1}\sum%
\limits_{j=M_{s}}^{M_{s+1}-1}A_{n-i+1}^{-\alpha -1}A_{m-j+1}^{-\beta
-1}D_{i}\left( u\right) D_{j}\left( v\right) \times \right.  \label{for7}
\end{equation}

\begin{equation*}
\left. \times \left[ f\left( \cdot -u,\cdot -v\right) -f\left( \cdot ,\cdot
\right) \right] d\mu \left( u,v\right) \right\Vert _{p}
\end{equation*}

\begin{equation*}
\leq \frac{1}{A_{n}^{-\alpha }A_{m}^{-\beta }}\text{ }\sum%
\limits_{r=0}^{k-2}\sum\limits_{s=0}^{l-2}\left\Vert \text{ }%
\int\limits_{G_{m}^{2}}\sum\limits_{i=M_{r}}^{M_{r+1}-1}\sum%
\limits_{j=M_{s}}^{M_{s+1}-1}A_{n-i+1}^{-\alpha -1}A_{m-j+1}^{-\beta
-1}D_{i}\left( u\right) D_{j}\left( v\right) \times \right.
\end{equation*}

\begin{equation*}
\left. \times \left[ f\left( \cdot -u,\cdot -v\right) -S_{M_{r,}M_{s}}\left(
\cdot -u,\cdot -v,f\right) \right] d\mu \left( u,v\right) \right\Vert _{p}
\end{equation*}

\begin{equation*}
+\frac{1}{A_{n}^{-\alpha }A_{m}^{-\beta }}\text{ }\sum\limits_{r=0}^{k-2}%
\sum\limits_{s=0}^{l-2}\left\Vert \text{ }\int\limits_{G_{m}^{2}}\sum%
\limits_{i=M_{r}}^{M_{r+1}-1}\sum\limits_{j=M_{s}}^{M_{s+1}-1}A_{n-i+1}^{-%
\alpha -1}A_{m-j+1}^{-\beta -1}D_{i}\left( u\right) D_{j}\left( v\right)
\times \right.
\end{equation*}

\begin{equation*}
\left. \times \left[ S_{M_{r,}M_{s}}\left( \cdot -u,\cdot -v,f\right)
-S_{M_{r,}M_{s}}\left( \cdot ,\cdot ,f\right) \right] d\mu \left( u,v\right)
\right\Vert _{p}
\end{equation*}

\begin{equation*}
+\frac{1}{A_{n}^{-\alpha }A_{m}^{-\beta }}\text{ }\sum\limits_{r=0}^{k-2}%
\sum\limits_{s=0}^{l-2}\left\Vert \text{ }\int\limits_{G_{m}^{2}}\sum%
\limits_{i=M_{r}}^{M_{r+1}-1}\sum\limits_{j=M_{s}}^{M_{s+1}-1}A_{n-i+1}^{-%
\alpha -1}A_{m-j+1}^{-\beta -1}D_{i}\left( u\right) D_{j}\left( v\right)
\times \right.
\end{equation*}

\begin{equation*}
\left. \times \left[ S_{M_{r,}M_{s}}\left( \cdot ,\cdot ,f\right) -f\left(
\cdot ,\cdot \right) \right] d\mu \left( u,v\right) \right\Vert
_{p}=I_{11}+I_{12}+I_{13}.
\end{equation*}

It is easy to show that

\begin{equation}
I_{12}=0.  \label{for9}
\end{equation}

Using Lemma \ref{L1} for $I_{11}$we can write 
\begin{equation}
I_{11}\leq \frac{1}{A_{n}^{-\alpha }A_{m}^{-\beta }}\text{ }%
\sum\limits_{r=0}^{k-2}\sum\limits_{s=0}^{l-2}\text{ }\int%
\limits_{G_{m}^{2}}\left\vert
\sum\limits_{i=M_{r}}^{M_{r+1}-1}\sum%
\limits_{j=M_{s}}^{M_{s+1}-1}A_{n-i+1}^{-\alpha -1}A_{m-j+1}^{-\beta
-1}D_{i}\left( u\right) D_{j}\left( v\right) \right\vert \times
\label{for10}
\end{equation}

\begin{equation*}
\times \left\Vert f\left( \cdot -u,\cdot -v\right) -S_{M_{r,}M_{s}}\left(
\cdot -u,\cdot -v,f\right) \right\Vert _{p}d\mu \left( u,v\right)
\end{equation*}

\begin{equation*}
\leq c\left( \alpha ,\beta \right) n^{\alpha }m^{\beta }\text{ }%
\sum\limits_{r=0}^{k-2}\sum\limits_{s=0}^{l-2}\left( \omega
_{1}(f,1/M_{r})_{p}+\omega _{2}(f,1/M_{s})_{p}\right) \times
\end{equation*}

\begin{equation*}
\times \left( \int\limits_{G_{m}}\left\vert
\sum\limits_{i=M_{r}}^{M_{r+1}-1}A_{n-i+1}^{-\alpha -1}D_{i}\left( u\right)
\right\vert d\mu \left( u\right) \right) \left(
\int\limits_{G_{m}}\left\vert
\sum\limits_{j=M_{s}}^{M_{s+1}-1}A_{m-j+1}^{-\beta -1}D_{j}\left( v\right)
\right\vert d\mu \left( v\right) \right)
\end{equation*}

\begin{equation*}
\leq c\left( \alpha ,\beta \right) n^{\alpha }m^{\beta }\text{ }%
\sum\limits_{r=0}^{k-2}\sum\limits_{s=0}^{l-2}\left( \omega
_{1}(f,1/M_{r})_{p}+\omega _{2}(f,1/M_{s})_{p}\right) \times
\end{equation*}

\begin{equation*}
\times \left( \sqrt{M_{r+1}}\left( \sum\limits_{i=M_{r}}^{M_{r+1}-1}\left(
n-i+1\right) ^{-2\alpha -2}\right) ^{1/2}\right) \times
\end{equation*}

\begin{equation*}
\times \left( \sqrt{M_{s+1}}\left( \sum\limits_{j=M_{s}}^{M_{s+1}-1}\left(
m-j+1\right) ^{-2\beta -2}\right) ^{1/2}\right)
\end{equation*}

\begin{equation*}
\leq c\left( \alpha ,\beta \right) n^{\alpha }m^{\beta
}\sum\limits_{r=0}^{k-2}\sum\limits_{s=0}^{l-2}\left( \omega
_{1}(f,1/M_{r})_{p}+\omega _{2}(f,1/M_{s})_{p}\right) \times
\end{equation*}

\begin{equation*}
\times \left( \sqrt{M_{r+1}}\left( n-M_{r+1}\right) ^{-\alpha -1}\sqrt{%
M_{r+1}}\right) \left( \sqrt{M_{s+1}}\left( n-M_{s+1}\right) ^{-\beta -1}%
\sqrt{M_{s+1}}\right)
\end{equation*}

\begin{equation*}
\leq c\left( \alpha ,\beta \right) n^{\alpha }m^{\beta
}\sum\limits_{r=0}^{k-2}\sum\limits_{s=0}^{l-2}\frac{M_{r+1}}{M_{k}^{\alpha
+1}}\frac{M_{s+1}}{M_{l}^{\beta +1}}\left( \omega _{1}(f,1/M_{r})_{p}+\omega
_{2}(f,1/M_{s})_{p}\right)
\end{equation*}

\begin{equation*}
\leq c\left( \alpha ,\beta \right) \left( \sum\limits_{r=0}^{k-2}\frac{M_{r}%
}{M_{k}}\omega _{1}\left( f,1/M_{r}\right) _{p}+\sum\limits_{s=0}^{l-2}\frac{%
M_{s}}{M_{l}}\omega _{2}\left( f,1/M_{s}\right) _{p}\right) .
\end{equation*}

\bigskip Analogously, we can prove that%
\begin{equation}
I_{13}\leq c\left( \alpha ,\beta \right) \left( \sum\limits_{r=0}^{k-2}\frac{%
M_{r}}{M_{k}}\omega _{1}\left( f,1/M_{r}\right) _{p}+\sum\limits_{s=0}^{l-2}%
\frac{M_{s}}{M_{l}}\omega _{2}\left( f,1/M_{s}\right) _{p}\right) .
\label{for11}
\end{equation}

Combining (\ref{for7})-(\ref{for11}) for $I_{1}$we recive that

\begin{equation}
I_{1}\leq c\left( \alpha ,\beta \right) \left( \sum\limits_{r=0}^{k-2}\frac{%
M_{r}}{M_{k}}\omega _{1}\left( f,1/M_{r}\right) _{p}+\sum\limits_{s=0}^{l-2}%
\frac{M_{s}}{M_{l}}\omega _{2}\left( f,1/M_{s}\right) _{p}\right) .
\end{equation}

For $I_{2}$ we can write%
\begin{equation}
I_{2}\leq \frac{1}{A_{n}^{-\alpha }A_{m}^{-\beta }}\left\Vert \text{ }%
\int\limits_{G_{m}^{2}}A_{m-M_{l-1}+1}^{-\beta }D_{M_{l-1}}\left( v\right)
\sum\limits_{i=1}^{M_{k-1}-1}A_{n-i+1}^{-\alpha -1}D_{i}\left( u\right)
\times \right.  \label{for12.1}
\end{equation}

\begin{equation*}
\left. \times \left[ f\left( \cdot -u,\cdot -v\right) -f\left( \cdot
-u,\cdot \right) \right] d\mu \left( u,v\right) \right\Vert _{p}
\end{equation*}

\begin{equation*}
+\frac{1}{A_{n}^{-\alpha }A_{m}^{-\beta }}\left\Vert \text{ }%
\int\limits_{G_{m}^{2}}A_{m-M_{l-1}+1}^{-\beta }D_{M_{l-1}}\left( v\right)
\sum\limits_{i=1}^{M_{k-1}-1}A_{n-i+1}^{-\alpha -1}D_{i}\left( u\right)
\times \right.
\end{equation*}

\begin{equation*}
\left. \times \left[ f\left( \cdot -u,\cdot \right) -f\left( \cdot ,\cdot
\right) \right] d\mu \left( u,v\right) \right\Vert _{p}=I_{21}+I_{22}.
\end{equation*}

From the generalized Minkowski inequality, and by (\ref{for1}) and (\ref%
{for4}) we obtain

\begin{equation}
I_{21}\leq c\left( \alpha ,\beta \right) \frac{M_{l-1}}{A_{n}^{-\alpha }}%
\int\limits_{I_{l-1}}\left( \int\limits_{G_{m}}\left\vert
\sum\limits_{i=1}^{M_{k-1}-1}A_{n-i+1}^{-\alpha -1}D_{i}\left( u\right)
\right\vert \right. \times  \label{for12.2}
\end{equation}

\begin{equation*}
\times \left. \left\Vert f\left( \cdot -u,\cdot -v\right) -f\left( \cdot
-u,\cdot \right) \right\Vert _{p}d\mu \left( u\right) \right) d\mu \left(
v\right)
\end{equation*}

\begin{equation*}
\leq c\left( \alpha ,\beta \right) n^{\alpha }\omega _{2}\left(
f,1/M_{l-1}\right) \left( \int\limits_{G_{m}}\left\vert
\sum\limits_{i=1}^{M_{k-1}-1}A_{n-i+1}^{-\alpha -1}D_{i}\left( u\right)
\right\vert d\mu \left( u\right) \right)
\end{equation*}

\begin{equation*}
\leq c\left( \alpha ,\beta \right) n^{\alpha }\omega _{2}\left(
f,1/M_{l-1}\right) \left( \sqrt{M_{k-1}}\left(
\sum\limits_{i=1}^{M_{k-1}-1}\left( n-i+1\right) ^{-2\alpha -2}\right)
^{1/2}\right)
\end{equation*}

\begin{equation*}
\leq c\left( \alpha ,\beta \right) n^{\alpha }\omega _{2}\left(
f,1/M_{l-1}\right) \left( \sqrt{M_{k-1}}\left( n-M_{k-1}\right) ^{-\alpha -1}%
\sqrt{M_{k-1}}\right)
\end{equation*}

\begin{equation*}
\leq c\left( \alpha ,\beta \right) \omega _{2}\left( f,1/M_{l-1}\right) .
\end{equation*}

The estimation of $I_{22}$ is analogous to the estimation of $I_{1}$ and we
have

\begin{equation}
I_{22}\leq c\left( \alpha ,\beta \right) \sum\limits_{r=0}^{k-2}\frac{M_{r}}{%
M_{k}}\omega _{1}\left( f,1/M_{r}\right) _{p}.  \label{for13}
\end{equation}

So, combining (\ref{for12.1})-(\ref{for13}) \ for $I_{2}$ we \ have

\begin{equation}
I_{2}\leq c\left( \alpha ,\beta \right) \left( \sum\limits_{r=0}^{k-2}\frac{%
M_{r}}{M_{k}}\omega _{1}\left( f,1/M_{r}\right) _{p}+\omega _{2}\left(
f,1/M_{l-1}\right) \right) .  \label{for13.1}
\end{equation}

The estimation $I_{3}$ is analogous to the estimation of $I_{2}$ and we have

\begin{equation}
I_{3}\leq c\left( \alpha ,\beta \right) \left( \sum\limits_{s=0}^{l-2}\frac{%
M_{s}}{M_{l}}\omega _{2}\left( f,1/M_{s}\right) _{p}+\omega _{1}\left(
f,1/M_{k-1}\right) \right) .  \label{for14}
\end{equation}

Combining (\ref{for5})-(\ref{for6}), (\ref{for11}), (\ref{for13.1})-(\ref%
{for14}) we receive the proof of Lemma \ref{L2}.\bigskip
\end{proof}

\begin{lemma}
\label{L3}Let $f$ $\in L^{p}(G_{m}^2)$ for some $p$ $\in $ $\left[ 1,\infty %
\right] .$ Then for every $\alpha ,\beta $ $\in $ $\left( 0,1\right) $ the
following estimations holds%
\begin{equation*}
II:=\frac{1}{A_{n}^{-\alpha }A_{m}^{-\beta }}\left\Vert \text{ }%
\int\limits_{G_{m}^{2}}\sum\limits_{i=M_{k-1}}^{n}\sum%
\limits_{j=0}^{M_{l-1}-1}A_{n-i}^{-\alpha }A_{m-j}^{-\beta }\psi _{i}\left(
u\right) \psi _{j}\left( v\right) \times \right.
\end{equation*}

\begin{equation*}
\left. \times \left[ f\left( \cdot -u,\cdot -v\right) -f\left( \cdot ,\cdot
\right) \right] d\mu \left( u\right) d\mu \left( v\right) \right\Vert
_{p}\leq c\left( \alpha ,\beta \right) \omega _{1}\left( f,1/M_{k-1}\right)
_{p}M_{k}^{\alpha },
\end{equation*}

\begin{equation*}
III:=\frac{1}{A_{n}^{-\alpha }A_{m}^{-\beta }}\left\Vert \text{ }%
\int\limits_{G_{m}^{2}}\sum\limits_{i=0}^{M_{k-1}-1}\sum%
\limits_{j=M_{l-1}}^{m}A_{n-i}^{-\alpha }A_{m-j}^{-\beta }\psi _{i}\left(
u\right) \psi _{j}\left( v\right) \times \right.
\end{equation*}

\begin{equation*}
\left. \times \left[ f\left( \cdot -u,\cdot -v\right) -f\left( \cdot ,\cdot
\right) \right] d\mu \left( u\right) d\mu \left( v\right) \right\Vert
_{p}\leq c\left( \alpha ,\beta \right) \omega _{2}\left( f,1/M_{l-1}\right)
_{p}M_{l}^{\beta }.
\end{equation*}

where $M_{k}\leq n<M_{k+1},M_{l}\leq m<M_{l+1}.$\newline
\end{lemma}

\begin{proof}[Proof of Lemma \protect\ref{L3}]
From the generalized Minkowski inequality we have

\begin{equation}
II=\frac{1}{A_{n}^{-\alpha }A_{m}^{-\beta }}\left\Vert \text{ }%
\int\limits_{G_{m}^{2}}\sum\limits_{i=M_{k-1}}^{n}\sum%
\limits_{j=0}^{M_{l-1}-1}A_{n-i}^{-\alpha }A_{m-j}^{-\beta }\psi _{i}\left(
u\right) \psi _{j}\left( v\right) \times \right.  \label{for21}
\end{equation}

\begin{equation*}
\left. \times f\left( \cdot -u,\cdot -v\right) d\mu \left( u,v\right)
\right\Vert _{p}
\end{equation*}

\begin{equation*}
=\frac{1}{A_{n}^{-\alpha }A_{m}^{-\beta }}\left\Vert \text{ }%
\int\limits_{G_{m}^{2}}\sum\limits_{i=M_{k-1}}^{n}\sum%
\limits_{j=0}^{M_{l-1}-1}A_{n-i}^{-\alpha }A_{m-j}^{-\beta }\psi _{i}\left(
u\right) \psi _{j}\left( v\right) \times \right.
\end{equation*}

\begin{equation*}
\left. \times \left[ f\left( \cdot -u,\cdot -v\right) -S_{M_{k-1}}^{\left(
1\right) }\left( \cdot -u,\cdot -v,f\right) \right] d\mu \left( u,v\right)
\right\Vert _{p}
\end{equation*}

\begin{equation*}
\leq \frac{1}{A_{n}^{-\alpha }A_{m}^{-\beta }}\int\limits_{G_{m}^{2}}\left%
\vert
\sum\limits_{i=M_{k-1}}^{M_{k}-1}\sum\limits_{j=0}^{M_{l-1}-1}A_{n-i}^{-%
\alpha }A_{m-j}^{-\beta }\psi _{i}\left( u\right) \psi _{j}\left( v\right)
\right\vert \times
\end{equation*}

\begin{equation*}
\times \left\Vert f\left( \cdot -u,\cdot -v\right) -S_{M_{k-1}}^{\left(
1\right) }\left( \cdot -u,\cdot -v,f\right) \right\Vert _{p}d\mu \left(
u,v\right)
\end{equation*}

\begin{equation*}
+\frac{1}{A_{n}^{-\alpha }A_{m}^{-\beta }}\int\limits_{G_{m}^{2}}\left\vert
\sum\limits_{i=M_{k}}^{n}\sum\limits_{j=0}^{M_{l-1}-1}A_{n-i}^{-\alpha
}A_{m-j}^{-\beta }\psi _{i}\left( u\right) \psi _{j}\left( v\right)
\right\vert \text{ }\times
\end{equation*}

\begin{equation*}
\times \left. \left\Vert f\left( \cdot -u,\cdot -v\right)
-S_{M_{k-1}}^{\left( 1\right) }\left( \cdot -u,\cdot -v,f\right) \right\Vert
_{p}d\mu \left( u\right) \right) d\mu \left( v\right) =II_{1}+II_{2}.
\end{equation*}%
In \cite{Te} present author showed that the inequality

\begin{equation}
\int\limits_{G_{m}}\left\vert
\sum\limits_{v=M_{k-1}}^{M_{k}-1}A_{n-v}^{-\alpha }\psi _{v}(u)\right\vert
d\mu (u)\leq c\left( \alpha \right) ,\ \ \ \left( k=1,2...\right)
\label{for22.1}
\end{equation}%
holds true.

Using Lemma \ref{L1}, by (\ref{for4}) and (\ref{for22.1}) for $II_{1}$ we
can write

\begin{equation}
II_{1}\leq c\left( \alpha ,\beta \right) n^{\alpha }m^{\beta }\omega
_{1}\left( f,1/M_{k-1}\right) _{p}\left( \int\limits_{G_{m}}\left\vert
\sum\limits_{i=M_{k-1}}^{M_{k}-1}A_{n-i}^{-\alpha }\psi _{i}\left( u\right)
\right\vert d\mu \left( u\right) \right) \times  \label{for23}
\end{equation}

\begin{equation*}
\times \left( \int\limits_{G_{m}}\left\vert
\sum\limits_{j=1}^{M_{l-1}}A_{m-j+1}^{-\beta }\psi _{j-1}\left( v\right)
\right\vert d\mu \left( v\right) \right)
\end{equation*}

\begin{equation*}
\leq c\left( \alpha ,\beta \right) n^{\alpha }m^{\beta }\omega _{1}\left(
f,1/M_{k-1}\right) _{p}\left( \sqrt{M_{l-1}}\left(
\sum\limits_{i=1}^{M_{l-1}}\left( m-j+1\right) ^{-2\beta -2}\right)
^{1/2}\right)
\end{equation*}

\begin{equation*}
\leq c\left( \alpha ,\beta \right) n^{\alpha }m^{\beta }\omega _{2}\left(
f,1/M_{k-1}\right) _{p}\left( \sqrt{M_{l-1}}\left( n-M_{l-1}\right) ^{-\beta
-1}\sqrt{M_{l-1}}\right)
\end{equation*}

\begin{equation*}
\leq c\left( \alpha ,\beta \right) \omega _{1}\left( f,1/M_{k-1}\right)
_{p}M_{k}^{\alpha }.
\end{equation*}%
The estimation of $II_{2}$ is analogous to the estimation of $II_{1}$ and we
have

\begin{equation}
II_{2}\leq c\left( \alpha ,\beta \right) \omega _{1}\left(
f,1/M_{k-1}\right) _{p}M_{k}^{\alpha }.  \label{for24}
\end{equation}

Combining (\ref{for21})-(\ref{for24}) we have%
\begin{equation}
II\leq c\left( \alpha ,\beta \right) \omega _{1}\left( f,1/M_{k-1}\right)
_{p}M_{k}^{\alpha }.  \label{for26}
\end{equation}

Analogously, we can prove that

\begin{equation}
III \label{for27.1}\leq c\left( \alpha ,\beta \right) \omega _{2}\left( f,1/M_{l-1}\right)
_{p}M_{l}^{\beta }.
\end{equation}

Combining (\ref{for26})-(\ref{for27.1}) we receive the proof of Lemma \ref%
{L3}.\bigskip
\end{proof}

\begin{lemma}
\label{L4}Let $f$ $\in L^{p}(G_{m}^2)$ for some $p$ $\in $ $\left[ 1,\infty %
\right] .$ Then for every $\alpha ,\beta $ $\in $ $\left( 0,1\right) $ the
following estimations holds
\end{lemma}

\begin{equation*}
IV:=\frac{1}{A_{n}^{-\alpha }A_{m}^{-\beta }}\left\Vert \text{ }%
\int\limits_{G_{m}^{2}}\sum\limits_{i=M_{k-1}}^{n}\sum%
\limits_{j=M_{l-1}}^{m}A_{n-i}^{-\alpha }A_{m-j}^{-\beta }\psi _{i}\left(
u\right) \psi _{j}\left( v\right) \times \right.
\end{equation*}

\begin{equation*}
\left. \times \left[ f\left( \cdot -u,\cdot -v\right) -f\left( \cdot ,\cdot
\right) \right] d\mu \left( u\right) d\mu \left( v\right) \right\Vert
_{p}\leq c\left( \alpha ,\beta \right) \omega _{1,2}\left(
f,1/M_{k},1/M_{l}\right) _{p}M_{k}^{\alpha }M_{l}^{\beta },
\end{equation*}

where $M_{k}\leq n<M_{k+1},M_{l}\leq m<M_{l+1}.$\newline

\begin{proof}[Proof of Lemma \protect\ref{L4}]
From the generalized Minkowski inequality, and by (\ref{for1}) and (\ref%
{for4}) we obtain

\begin{equation}
IV=\frac{1}{A_{n}^{-\alpha }A_{m}^{-\beta }}\left\Vert \text{ }%
\int\limits_{G_{m}^{2}}\sum\limits_{i=M_{k-1}}^{n}\sum%
\limits_{j=M_{l-1}}^{m}A_{n-i}^{-\alpha }A_{m-j}^{-\beta }\psi _{i}\left(
u\right) \psi _{j}\left( v\right) \times \right.  \label{for15}
\end{equation}

\begin{equation*}
\left. \times f\left( \cdot -u,\cdot -v\right) d\mu \left( u,v\right)
\right\Vert _{p}
\end{equation*}

\begin{equation*}
\leq \frac{1}{A_{n}^{-\alpha }A_{m}^{-\beta }}\left\Vert \text{ }%
\int\limits_{G_{m}^{2}}\sum\limits_{i=M_{k-1}}^{n}\sum%
\limits_{j=M_{l-1}}^{m}A_{n-i}^{-\alpha }A_{m-j}^{-\beta }\psi _{i}\left(
u\right) \psi _{j}\left( v\right) \times \right.
\end{equation*}

\begin{equation*}
\times \left[ S_{M_{k-1}{},M_{l-1}}\left( \cdot -u,\cdot -v,f\right)
-S_{M_{k-1}}^{\left( 1\right) }\left( \cdot -u,\cdot -v,f\right) \right.
\end{equation*}

\begin{equation*}
\left. \left. -S_{M_{l-1}}^{\left( 2\right) }\left( \cdot -u,\cdot
-v,f\right) +f\left( \cdot -u,\cdot -v\right) \right] d\mu \left( u,v\right)
\right\Vert _{p}
\end{equation*}

\begin{equation*}
\leq \frac{1}{A_{n}^{-\alpha }A_{m}^{-\beta }}\int\limits_{G_{m}^{2}}\left%
\vert \sum\limits_{i=M_{k-1}}^{n}\sum\limits_{j=M_{l-1}}^{m}A_{n-i}^{-\alpha
}A_{m-j}^{-\beta }\psi _{i}\left( u\right) \psi _{j}\left( v\right)
\right\vert \times
\end{equation*}

\begin{equation*}
\times \left\Vert S_{M_{k-1}{},M_{l-1}}\left( \cdot -u,\cdot -v,f\right)
-S_{M_{k-1}}^{\left( 1\right) }\left( \cdot -u,\cdot -v,f\right) \right.
\end{equation*}

\begin{equation*}
\left. -S_{M_{l-1}}^{\left( 2\right) }\left( \cdot -u,\cdot -v,f\right)
+f\left( \cdot -u,\cdot -v\right) \right\Vert _{p}d\mu \left( u,v\right)
\end{equation*}

\begin{equation*}
\leq c\left( \alpha ,\beta \right) n^{\alpha }m^{\beta }\omega _{1,2}\left(
f,1/M_{k-1},1/M_{l-1}\right) _{p}
\end{equation*}

\begin{equation*}
\times \int\limits_{G_{m}^{2}}\left\vert \text{ }\sum\limits_{i=M_{k-1}}^{n}%
\sum\limits_{j=M_{l-1}}^{m}A_{n-i}^{-\alpha }A_{m-j}^{-\beta }\psi
_{i}\left( u\right) \psi _{j}\left( v\right) \right\vert d\mu \left(
u,v\right)
\end{equation*}

\begin{equation*}
\leq c\left( \alpha ,\beta \right) n^{\alpha }m^{\beta }\omega _{1,2}\left(
f,1/M_{k-1},1/M_{l-1}\right) _{p}
\end{equation*}

\begin{equation*}
\leq c\left( \alpha ,\beta \right) M_{k}^{\alpha }M_{l}^{\beta }\omega
_{1,2}\left( f,1/M_{k-1},1/M_{l-1}\right) _{p}.
\end{equation*}

Lemma \ref{L4} is proved.
\end{proof}

\begin{proof}[Proof of Theorem \protect\ref{T3}]
It is evident that

\begin{equation*}
\sigma _{n,m}^{-\alpha ,-\beta }\left( f,x,y\right) -f\left( x,y\right) =%
\frac{1}{A_{n}^{-\alpha }A_{m}^{-\beta }}\int\limits_{G_{m}^{2}}\sum%
\limits_{i=0}^{M_{k-1}-1}\sum\limits_{j=0}^{M_{l-1}-1}A_{n-i}^{-\alpha
}A_{m-j}^{-\beta }\psi _{i}\left( u\right) \psi _{j}\left( v\right) \times
\end{equation*}

\begin{equation*}
\times \left[ f\left( \cdot -u,\cdot -v\right) -f\left( \cdot,\cdot \right) \right] d\mu \left(
u,v\right)
\end{equation*}

\begin{equation*}
+\frac{1}{A_{n}^{-\alpha }A_{m}^{-\beta }}\text{ }\int\limits_{G_{m}^{2}}%
\sum\limits_{i=M_{k-1}}^{n}\sum\limits_{j=0}^{M_{l-1}-1}A_{n-i}^{-\alpha
}A_{m-j}^{-\beta }\psi _{i}\left( u\right) \psi _{j}\left( v\right) \times
\end{equation*}

\begin{equation*}
\times \left[ f\left( \cdot -u,\cdot -v\right) -f\left( \cdot ,\cdot \right) \right] d\mu \left(
u,v\right)
\end{equation*}

\begin{equation*}
+\frac{1}{A_{n}^{-\alpha }A_{m}^{-\beta }}\int\limits_{G_{m}^{2}}\sum%
\limits_{i=0}^{M_{k-1}-1}\sum\limits_{j=M_{l-1}}^{m}A_{n-i}^{-\alpha
}A_{m-j}^{-\beta }\psi _{i}\left( u\right) \psi _{j}\left( v\right) \times
\end{equation*}

\begin{equation*}
\times \left[ f\left( \cdot -u,\cdot -v\right) -f\left( \cdot ,\cdot \right) \right] d\mu \left(
u,v\right)
\end{equation*}

\begin{equation*}
+\frac{1}{A_{n}^{-\alpha }A_{m}^{-\beta }}\text{ }\int\limits_{G_{m}^{2}}%
\sum\limits_{i=M_{k-1}}^{n}\sum\limits_{j=M_{l-1}}^{m}A_{n-i}^{-\alpha
}A_{m-j}^{-\beta }\psi _{i}\left( u\right) \psi _{j}\left( v\right) \times
\end{equation*}

\begin{equation*}
\times \left[ f\left( \cdot -u,\cdot -v\right) -f\left(\cdot ,\cdot \right) \right] d\mu \left(
u\right) d\mu \left( v\right) =I+II+III+IV.
\end{equation*}%
Since%
\begin{equation*}
\left\Vert \sigma _{n,m}^{-\alpha ,-\beta }\left( f,x\right) -f\left(
x\right) \right\Vert _{p}\leq \left\Vert I\right\Vert _{p}+\left\Vert
II\right\Vert _{p}+\left\Vert III\right\Vert _{p}+\left\Vert IV\right\Vert
_{p}
\end{equation*}

From Lemmas \ref{L2}-\ref{L4} the proof of theorem is complete.
\end{proof}

\begin{proof}[Proof of Corollary \protect\ref{C2}]
Since

\begin{equation*}
\omega _{i}\left( f,\frac{1}{M_{n}}\right) \leq \omega \left( f,\frac{1}{%
M_{n}}\right) ,\text{ \ }\ i=1,2,  
\end{equation*}

\begin{equation*}
\omega _{1,2}\left( f,\frac{1}{M_{n}},\frac{1}{M_{m}}\right) \leq 2\omega
_{1}\left( f,\frac{1}{M_{n}}\right)
\end{equation*}

and%
\begin{equation*}
\omega _{1,2}\left( f,\frac{1}{M_{n}},\frac{1}{M_{m}}\right) \leq 2\omega
_{2}\left( f,\frac{1}{M_{m}}\right) ,
\end{equation*}

 we obtain

\begin{equation*}
\omega _{1,2}\left( f,\frac{1}{M_{n}},\frac{1}{M_{m}}\right) =\left( \omega
_{1,2}\left( f,\frac{1}{M_{n}},\frac{1}{M_{m}}\right) \right) ^{\frac{\alpha 
}{\alpha +\beta }}\left( \omega _{1,2}\left( f,\frac{1}{M_{n}},\frac{1}{M_{m}%
}\right) \right) ^{\frac{\beta }{\alpha +\beta }}  \label{1.12}
\end{equation*}

\begin{equation*}
\leq 2\left( \omega _{1}\left( f,\frac{1}{M_{n}}\right) \right) ^{\frac{%
\alpha }{\alpha +\beta }}\left( \omega _{2}\left( f,\frac{1}{M_{m}}\right)
\right) ^{\frac{\beta }{\alpha +\beta }}
\end{equation*}

\begin{equation*}
\leq 2\left( \omega \left( f,\frac{1}{M_{n}}\right) \right) ^{\frac{\alpha }{%
\alpha +\beta }}\left( \omega \left( f,\frac{1}{M_{m}}\right) \right) ^{%
\frac{\beta }{\alpha +\beta }}.
\end{equation*}

The validity of Corollary 2 follows immediately from Corollary 1.
\end{proof}

\begin{proof}[Proof of Theorem \protect\ref{T4}]
First, we set%
\begin{equation*}
f_{j}\left( x\right) =\rho _{j}\left( x\right) =\exp \frac{2\pi ix_{j}}{m_{j}%
}.
\end{equation*}

Then we define the function%
\begin{equation*}
f\left( x,y\right) =\sum\limits_{j=1}^{\infty }\frac{1}{M_{_{j}}^{\left(
\alpha +\beta \right) }}f_{j}\left( x\right) f_{j}\left( y\right) .
\end{equation*}

First, we prove that%
\begin{equation}
\omega \left( f,\frac{1}{M_{n}}\right) _{C}=O\left( \left( \frac{1}{M_{n}}%
\right) ^{\alpha +\beta }\right) .  \label{1.13}
\end{equation}

Since%
\begin{equation*}
\left\vert f_{j}\left( x-t\right) -f_{j}\left( x\right) \right\vert =0,  \text{
\ \ \ }j=0,1,...,n-1,\text{
\ \  \ }t\in I_{n}
\end{equation*}

we find%
\begin{equation*}
\left\vert f\left( x-t,y\right) -f\left( x,y\right) \right\vert \leq
\sum\limits_{j=1}^{n-1}\frac{1}{M_{_{j}}^{\left( \alpha +\beta \right) }}%
\left\vert f_{j}\left( x-t\right) -f_{j}\left( x\right) \right\vert
\end{equation*}

\begin{equation*}
+\sum\limits_{j=n}^{\infty }\frac{2}{M_{_{j}}^{\left( \alpha +\beta \right) }%
}\leq \frac{c}{M_{n}^{\left( \alpha +\beta \right) }}.
\end{equation*}

Hence%
\begin{equation}
\omega _{1}\left( f,\frac{1}{M_{n}}\right) =O\left( \left( \frac{1}{M_{n}}%
\right) ^{\alpha +\beta }\right) .  \label{1.14}
\end{equation}

Analogously, we have%
\begin{equation}
\omega _{2}\left( f,\frac{1}{M_{m}}\right) =O\left( \left( \frac{1}{M_{m}}%
\right) ^{\alpha +\beta }\right) .  \label{1.15}
\end{equation}

Now, by $\left( \ref{1.14}\right) $ and $\left( \ref{1.15}\right) $, we
obtain $\left( \ref{1.13}\right) $.

Next, we shall prove that $\sigma _{M_{n},M_{n}}^{-\alpha ,-\beta }\left(
f\right) $ diverge in the metric of $L^{1}.$ It is clear that

\begin{equation}
\left\Vert \sigma _{M_{n},M_{n}}^{-\alpha ,-\beta }\left( f\right)
-f\right\Vert _{1}\geq \left\vert \int\limits_{G_{m}^{2}}\left[ \sigma
_{M_{n},M_{n}}^{-\alpha ,-\beta }\left( f;x,y\right) -f\left( x,y\right) %
\right] \psi _{M_{k}}\left( x\right) \psi _{M_{k}}\left( y\right) d\mu
\left( x,y\right) \right\vert  \label{1.16}
\end{equation}

\begin{equation*}
\geq \left\vert \int\limits_{G_{m}^{2}}\sigma _{M_{n},M_{n}}^{-\alpha
,-\beta }\left( f;x,y\right) \psi _{M_{k}}\left( x\right) \psi
_{M_{k}}\left( y\right) dxdy\right\vert -\left\vert \widehat{f}\left(
M_{k},M_{k}\right) \right\vert
\end{equation*}

\begin{equation*}
=\left\vert \frac{1}{A_{M_{k}}^{-\alpha }A_{M_{k}}^{-\beta }}%
\sum\limits_{i=0}^{M_{l_{k}}}\sum\limits_{j=0}^{M_{l_{k}}}A_{M_{k}-i}^{-%
\alpha }A_{M_{k}-j}^{-\beta }\hat{f}\left( i,j\right)
\int\limits_{G_{m}^{2}}\psi _{i}\left( x\right) \psi _{j}\left( y\right)
\psi _{M_{k}}\left( x\right) \psi _{M_{k}}\left( y\right) d\mu \left(
x,y\right) \right\vert
\end{equation*}

\begin{equation*}
-\left\vert \widehat{f}\left( M_{k},M_{k}\right) \right\vert =\frac{1}{%
A_{M_{l_{k}}}^{-\alpha }A_{M_{l_{k}}}^{-\beta }}\left\vert \widehat{f}\left(
M_{k},M_{k}\right) \right\vert -\left\vert \widehat{f}\left(
M_{k},M_{k}\right) \right\vert .
\end{equation*}

We have%
\begin{equation*}
\widehat{f}\left( M_{k},M_{k}\right) =\int\limits_{G_{m}^{2}}f\left(
x,y\right) \psi _{M_{k}}\left( x\right) \psi _{M_{k}}\left( y\right) d\mu
\left( x,y\right)
\end{equation*}

\begin{equation*}
=\sum\limits_{j=1}^{\infty }\frac{1}{M_{j}^{\left( \alpha +\beta \right) }}%
\int\limits_{G_{m}^{2}}\rho _{j}\left( x\right) \rho _{j}\left( y\right)
\psi _{M_{k}}\left( x\right) \psi _{M_{k}}\left( y\right) d\mu \left(
x,y\right)
\end{equation*}

\begin{equation*}
=\sum\limits_{j=1}^{\infty }\frac{1}{M_{j}^{\left( \alpha +\beta \right) }}%
\int\limits_{G_{m}}\rho _{j}\left( x\right) \psi _{M_{k}}\left( x\right)
d\mu \left(
x\right)\int\limits_{G_{m}}\rho _{j}\left( y\right) \psi _{M_{k}}\left( y\right)
d\mu \left(
y\right)=\frac{1}{M_{k}^{\left( \alpha +\beta \right) }}.
\end{equation*}

So, we can write

\begin{equation}
\left\Vert \sigma _{M_{n},M_{n}}^{-\alpha ,-\beta }\left( f\right)
-f\right\Vert _{1}\geq c\left( \alpha ,\beta \right) .  \label{7}
\end{equation}

Theorem \ref{T4} is proved.
\end{proof}

\end{document}